\documentclass[12pt]{article}
\usepackage{graphics,graphicx}

\usepackage[T1]{fontenc}
\usepackage[utf8]{inputenc}
\usepackage{mathtools}   

\usepackage{amsmath}
\usepackage{amssymb}
\usepackage{amsfonts}
\usepackage{float}
\usepackage{caption}
\usepackage{bbold}
\usepackage{amsthm}
\makeatletter
\def\ft{\@ifnextchar[{\ft@s}{\ft@}}
\def\ft@{\ft@@@s[\f@size]}
\def\ft@s[{\@ifnextchar{a}{\ft@sz[}{\ft@@s[}}
\def\ft@@s[{\@ifnextchar{s}{\ft@sz[}{\ft@@@s[}}
\def\ft@@@s[#1]{\ft@sz[at #1pt]}
\def\ft@sz[#1]#2{\font\fonttemp=#2 #1\fonttemp\ignorespaces}
\makeatother

\def\ns{\fontshape{\shapedefault}\selectfont}

\let\mycal\cal
\def\cal#1{{\mycal #1}}

\let\texbf\bf
\def\bf{\texbf\boldmath}

\makeatother

\makeatletter

\def\@@bold{bold}

\def\widebar{\ifx\math@version\@@bold
  \let\@widebar\@@@widebar\else\let\@widebar\@@widebar\fi\@widebar}

\def\@@widebar#1{\text{\setbox15\hbox{$#1$}%
  \dimen15 0.45\wd15\advance\dimen15 0.15\ht15%
  \dimen16\ht15\advance\dimen16 0.00em\advance\dimen16 0.3ex%
  \dimen17 0.65\wd15\advance\dimen17 0.05\ht15\advance\dimen17 0.1ex%
  \dimen18 0.035em\advance\dimen18 0.00ex
  \put[\dimen15,\dimen16][c]{\vrule depth 0pt height \dimen18 width \dimen17}}#1}

\def\@@@widebar#1{\text{\setbox15\hbox{$#1$}%
  \dimen15 0.45\wd15\advance\dimen15 0.15\ht15%
  \dimen16\ht15\advance\dimen16 0.00em\advance\dimen16 0.26ex%
  \dimen17 0.65\wd15\advance\dimen17 0.05\ht15\advance\dimen17 0.1ex%
  \dimen18 0.05em\advance\dimen18 0.00ex
  \put[\dimen15,\dimen16][c]{\vrule depth 0pt height \dimen18 width \dimen17}}#1}
 
 \makeatother

\makeatletter

\def\put{\@ifnextchar[{\@put}{\@@rput[\z@,\z@][r]}}
\def\@put[#1]{\@ifnextchar[{\@@put[#1]}{\@@@@@put[#1]}}
\def\@@put[#1][{\@ifnextchar{l}{\@@lput[#1][}{\@@@put[#1][}}
\def\@@@put[#1][{\@ifnextchar{c}{\@@cput[#1][}{\@@@@put[#1][}}
\def\@@@@put[#1][{\@ifnextchar{r}{\@@rput[#1][}{\relax}}
\def\@@@@@put[{\@ifnextchar{l}{\@@lput[\z@,\z@][}{\@@@@@@put[}}
\def\@@@@@@put[{\@ifnextchar{c}{\@@cput[\z@,\z@][}{\@@@@@@@put[}}
\def\@@@@@@@put[{\@ifnextchar{r}{\@@rput[\z@,\z@][}{\@@@@@@@@put[}}
\def\@@@@@@@@put[#1]{\@@rput[#1][r]}

\let\hm@d@\leavevmode

\long\def\@@lput[#1,#2][l]#3{\setbox0\hbox{#3}\hm@d@\raise#2\hbox to\z@{\dimen0 #1%
  \advance\dimen0-\wd0\kern\dimen0\dp0\z@\ht0\z@\wd0\z@\box0\hss}\ignorespaces}
\long\def\@@cput[#1,#2][c]#3{\setbox0\hbox{#3}\hm@d@\raise#2\hbox to\z@{\dimen0 #1%
  \advance\dimen0-.5\wd0\kern\dimen0\dp0\z@\ht0\z@\wd0\z@\box0\hss}\ignorespaces}
\long\def\@@rput[#1,#2][r]#3{\setbox0\hbox{\kern#1\raise#2\hbox{#3}}%
  \dp0\z@\ht0\z@\wd0\z@\hm@d@\box0\ignorespaces}

\def\flbox{\@ifnextchar[{\@flbox}{\@@rflbox[\z@,\z@][r]}}
\def\@flbox[#1]{\@ifnextchar[{\@@flbox[#1]}{\@@@@@flbox[#1]}}
\def\@@flbox[#1][{\@ifnextchar{l}{\@@lflbox[#1][}{\@@@flbox[#1][}}
\def\@@@flbox[#1][{\@ifnextchar{c}{\@@cflbox[#1][}{\@@@@flbox[#1][}}
\def\@@@@flbox[#1][{\@ifnextchar{r}{\@@rflbox[#1][}{\relax}}
\def\@@@@@flbox[{\@ifnextchar{l}{\@@lflbox[\z@,\z@][}{\@@@@@@flbox[}}
\def\@@@@@@flbox[{\@ifnextchar{c}{\@@cflbox[\z@,\z@][}{\@@@@@@@flbox[}}
\def\@@@@@@@flbox[{\@ifnextchar{r}{\@@rflbox[\z@,\z@][}{\@@@@@@@@flbox[}}
\def\@@@@@@@@flbox[#1]{\@@rflbox[#1][r]}
\long\def\@@lflbox[#1,#2][l]#3{\@@lput[#1,#2][l]{%
  \vtop{\leftskip\z@\parindent\z@\raggedleft\hm@d@#3}}}
\long\def\@@cflbox[#1,#2][c]#3{\@@cput[#1,#2][c]{%
  \vtop{\leftskip\z@\parindent\z@\raggedcenter\hm@d@#3}}}
\long\def\@@rflbox[#1,#2][r]#3{\@@rput[#1,#2][r]{%
  \vtop{\leftskip\z@\parindent\z@\raggedright\hm@d@#3}}}

\makeatother


\def\({\gdef\nextsp{}\mbox\bgroup{\ns(}\sl\aux}
\def\aux#1{\def\tempx{#1}\let\next\aux%
   \if\tempx)\let\next\egroup{\/\nextsp\ns)}%
   \else\if\tempx f{\nextsp f\gdef\nextsp{\kern0.2ex}}%
   \else\if\tempx i{\nextsp\kern0.1ex i\gdef\nextsp{\kern0.1ex}}%
   \else\if\tempx j{\nextsp\kern0.2ex j\gdef\nextsp{\kern0.1ex}}%
   \else\if\tempx l{\nextsp\kern0.1ex l\gdef\nextsp{\kern0.2ex}}%
   \else\if\tempx I{\nextsp I\gdef\nextsp{\kern0.2ex}}%
   \else\if\tempx'{\/$\mkern0.5mu'$\gdef\nextsp{}}%
   \else\if\tempx-{\/\ns-\gdef\nextsp{}}%
   \else\nextsp\tempx\gdef\nextsp{}%
   \fi\fi\fi\fi\fi\fi\fi\fi\next}

\newcommand{\fig}[2]{\includegraphics[scale=0.83]{Figures/#2}}

\newtheorem{theorem}{Theorem}[section]
\newtheorem{lemma}[theorem]{Lemma}

\newtheorem{proposition}[theorem]{Proposition}
\theoremstyle{definition} \newtheorem{definition}[theorem]{Definition}

\def\varemptyset{%
   {\text{\raise.21ex\hbox{$\not$}}\mkern.15mu\mathrm{O}\mkern.15mu}}
   
\def\widebar#1{\text{\setbox15\hbox{$#1$}%
  \dimen15 0.45\wd15\advance\dimen15 0.15\ht15%
  \dimen16\ht15\advance\dimen16 0.00em\advance\dimen16 0.3ex%
  \dimen17 0.65\wd15\advance\dimen17 0.05\ht15\advance\dimen17 0.1ex%
  \dimen18 0.03em\advance\dimen18 0.005em
  \put[\dimen15,\dimen16][c]{\vrule depth 0pt height \dimen18 width \dimen17}}#1}

\def\_{{\hbox to 1.2ex{\hss\vrule width1ex height0pt depth.4pt\hss}}}

\let\varemptyset\varemptyset

\let\bar\widebar
\let\tilde\tilde
\let\tilde\widetilde

\newcommand\red%
  {\mkern0.2mu\text{\ft{msam10}9\kern-0.2em9\kern-0.325em\ft{lasy10})}\mkern0.4mu}

\newcommand{\id}{\mathop{\mathrm{id}}\nolimits}
\newcommand{\one}{{\mathbb{1}}}

\newcommand{\diam}{\mathrel{\diamond}}

\newcommand{\C}{{\cal C}}

\renewcommand{\H}{{\cal H}}

\newcommand{\Alg}{\cal Alg}
\newcommand{\alg}{\overline{\cal Alg}}

\def\bPhi{\mkern-1mu\bar{\mkern1mu \Phi\mkern-2mu}\mkern2mu}

\makeatletter
\def\up{\@ifnextchar[{\@up}{\mathop{\uparrow}\nolimits}}
\def\@up[#1]{{\uparrow}\text{\raise .6ex\hbox{$_#1$}}}
\def\down{\@ifnextchar[{\@down}{\mathop{\downarrow}\nolimits}}
\def\@down[#1]{{\downarrow}\text{\raise .6ex\hbox{$_#1$}}}
\makeatother

\newcommand{\bs}{\text{\raise.4ex\hbox{\bf$\scriptscriptstyle\backslash$}}}

\newcommand{\beqn}{\begin{eqnarray*}}
\newcommand{\eeqn}{\end{eqnarray*}}
\newcommand{\be}{\begin{equation}}
\newcommand{\ee}{\end{equation}}
\newcommand{\ba}{\begin{array}}
\newcommand{\ea}{\end{array}}

\newcommand{\Chb}{{\cal C\mkern-1.8mu h\mkern-0.3mu b}}

\newcommand{\Cob}{{\cal C\mkern-1.1mu o\mkern-0.2mu b}}

\newcommand{\Cobt}{\tilde{\vrule height1.4ex width0pt \smash{\Cob}}{}}

\newcommand{\Hbb}{{\mkern1mu\bar{\bar{\mkern-1mu{\H}}}}{}}

\newcommand{\bbPhi}{\mkern-1.27mu\bar{\mkern1.27mu{\bPhi}\mkern-2.03mu}\mkern2.03mu}

\newcommand{\xdownarrow}[1]{%
  {\left\downarrow\vbox to #1{}\right.\kern-\nulldelimiterspace}
}

\newcommand{\xuparrow}[1]{%
  {\left\uparrow\vbox to #1{}\right.\kern-\nulldelimiterspace}
}

\begin{document}
\title{On the algebraic characterization of the category of 3-dimensional cobordisms}

\author{Ivelina Bobtcheva}

\maketitle

\begin{abstract}
It is proved in \cite{BP} that  the category of relative 3-dimensional cobordisms $\Cob^{2+1}$  is equivalent to the universal algebraic category $\Hbb^r$ generated by a Hopf algebra object. A different algebraic category $\alg$ generated by a Hopf algebra object is defined in \cite{AS} and it conjectured to be equivalent to $\Cob^{2+1}$ as well. We  prove that there exists a functor $\alg\to \Hbb^r$, and use it  to present an alternative set of axioms for $\Hbb^r$. 
\end{abstract}


\section{Introduction}

The development of quantum topology builds a bridge between low dimensional topology and the theory of Hopf algebras. The initial steps of the theory consisted in constructing invariants of 3-manifolds and 3-dimensional TQFT's, starting with the category of representations of some quantum groups \cite{RT91}. Then Hennings \cite{H96}, Kaufmann and Radford \cite{KR95} build similar invariants  directly from the Hopf algebra itself. Eventually, the works of Crane and Yetter \cite{CY99}, Habiro \cite{Ha00}, Kerler and Lyubashenko \cite{KL01} lead to the understanding that the  Hopf algebra structure is intrinsic to the category of relative 3-dimensional cobordisms $\Cob^{2+1}$. In particular, in \cite{Ke02}, Kerler  defined a universal monoidal braided category $\Alg$, freely generated by a Hopf algebra object and a full (surjective) functor from $\Alg$ to the category of relative 3-dimensional cobordisms\footnote{Actually, Kerler defines a functor from $\Alg$ to the category $\Cobt^{2+1}$ of 2-framed relative 3-dimensional cobordims. The difference between the two categories is that $\Cob^{2+1}$ is the quotient of $\Cobt^{2+1}$ by one more normalization relation (relation \(n) in Table \ref{table-Hrb/fig}). We  work here with $\Cob^{2+1}$ instead of $\Cobt^{2+1}$, because this is the framework in \cite{AS}. Observe that relation (25) of the algebra defined in \cite{AS} doesn't hold in $\Cobt^{2+1}$.}. He also posed the problem  \cite[Problem 8-16 (1)]{Oh02} to find a set of additional relations for $\Alg$, such that the above functor induces a category equivalence on the quotient of $\Alg$ by the new relations. 

\medskip
This problem was solved in \cite{BP}, where it was proved that in order to obtain a category equivalence, two more relations should be added to the defining axioms of $\Alg$ (see \(r8) and \(r9) in Table \ref{table-Hr-axioms/fig}). The first of these relations  can be thought as compatibility condition between the ribbon morphism and  the comultiplication, while second one can be thought as compatibility condition between the copairing  and the braiding morphism.  

In \cite{BP} the algebraic structure of $\Cob^{2+1}$ is derived from the one of the category of relative cobordisms of  4-dimensional 2-handlebodies $\Chb^{3+1}$. Indeed, any 3-dimensional cobordism can be seen as the  boundary of a 4-dimensional handlebody and this allows to define  the category  $\Cob^{2+1}$  as quotient of  $\Chb^{3+1}$, modulo all transformations which change the interior of the handlebody but leave invariant its boundary.  
Then two universal monoidal braided categories freely generated by a braided ribbon Hopf algebra object are introduced: $\H^r$ and its quotient $\Hbb^r$, and it is shown that we have the following commutative diagram: 
$$
\begin{array}{ccc}
\Chb^{3+1} &\longrightarrow &\Cob^{2+1}\\
\Phi\xuparrow{15 pt} && \bbPhi\xuparrow{15 pt} \\
\H^r &\longrightarrow & \Hbb^r 
\end{array},
$$
where the horizontal arrows represent quotient functors, while $\Phi$ and $\bbPhi$ are equivalences of categories. 

On the other hand,  in \cite{AS} the universal algebraic category $\alg$ is defined and it is announced an unpublished result of Habiro  \cite[Theorem 2.3]{AS}, according to which there is a category equivalence between the category of 3-dimensional cobordisms and $\alg$. 
The main difference in the elementary morphisms of $\alg$ and $\Hbb^r$ is that  while  both $\Alg$ and $\Hbb^r$  introduce the two sided integrals as elementary morphisms, this is not done in $\alg$. Moreover, in $\alg$ four new relations, regarding the ribbon structure, are added to the ones of Kerler's category $\Alg$ ((see ${\mathtt h10-13}$) in Table \ref{table-Alg-axioms/fig})  and  they are all different from relations \(r8-9), \(d) and \(n) in the definition of $\Hbb^r$.  

\smallskip
The goal of these notes is to summarize the results in \cite{BP}, related to the algebraic structures of the cobordism categories,  in order to facilitate their use and  to investigate the independence of the axioms of $\H^r$ and $\Hbb^r$ and  the relationship between the categories $\Hbb^r$ and $\alg$. 

In particular, we  show that in $\H^r$ the integral axioms and axioms \(r8) and \(r9)   are independent from the rest of axioms of the category (Proposition \ref{indep-axioms/prop}). Then we prove that there exists a functor $\Gamma: \alg\to \Hbb^r$  (Theorem \ref{functor-Hab/thm}). Finally, we present an alternative solution of the problem of Kerler by proving that in $\Hbb^r$ axiom \(r9)  is equivalent to $\Gamma({\mathtt h10})$, while axiom \(r8) of $\Hbb^r$ is equivalent to  the requirement that the adjoint morphism intertwines with the copairing (Proposition \ref{new-axiomsHbb/thm}). The  equivalence of $\Gamma({\mathtt h10})$ and \(r9) has been communicated to the author  by Mariya Stamatova. 

\smallskip
The author is grateful to Mariya Stamatova for posing some questions, which became the main motivation behind writing down the present notes, and  to Anna Beliakova and Marco De Renzi for their interest and encouragement and for pointing out some misprints and oversights in its previous version. The study of the relationship between the categories $\Hbb^r$ and $\alg$  has been  completed in \cite{BBDP}  where it is proven that the two categories are actually equivalent. 


\newpage
\section{Preliminaries}

In what follows we are going to use the terminology and the notations in \cite{BP}, where complete definitions, proofs and references can be found. The morphisms in the topological categories $\Chb^{3+1}$ and $\Cob^{2+1}$ will be described in terms of Kirby tangles, while the morphisms in the algebraic categories $\H^r$ and $\Hbb^r$ will be described in term of plane diagrams, where each diagram is a compositions of products of elementary ones and identities modulo some defining relations to which we refer as {\it axioms}. The composition of two diagrams $D_2\circ D_1$ is obtained by stacking $D_2$ on the top of $D_1$, while the product $D_1 \diam D_2$ is given by the horizontal juxtaposition of $D_1$ and $D_2$.

\subsection{The universal ribbon Hopf algebra $\H^r$}
\label{defn-Hr/sec}

\begin{definition}\label{cat-hopf-algebra/def}
Given a braided monoidal category $\C$ with product $\diam$, unit object $\one$ and braiding morphism $\gamma$, a \textit{braided Hopf algebra} in $\C$ is a objects $H$ in $\C$, equipped with the
following morphisms in $\C\,$:

\smallskip\noindent
a \textit{comultiplication} $\Delta : H \to H \diam H$,
such that:
\vskip-4pt
$$ (\Delta\diam\id_H) \circ \Delta = (\id_H\diam\Delta) \circ \Delta;
\eqno{\(a1)}$$
\vskip4pt

\smallskip\noindent
a \textit{counit} $\epsilon : H \to \one$, such that:
\vskip-4pt
$$(\epsilon\diam\id_H)\circ\Delta = \id_H = (\id_H\diam\epsilon)\circ\Delta;
\eqno{\(a2-2')}$$
\vskip4pt

\smallskip\noindent
a \textit{multiplication} $m : H \diam H$, such 
that:
\vskip-4pt
$$m \circ (m \diam \id_H) = m \circ (\id_H \diam m),
\eqno{\(a3)}$$
$$(m \diam m) \circ (\id_H \diam \gamma\diam \id_H) \circ
(\Delta \diam \Delta) = \Delta \circ m,
\eqno{\(a5)}$$
$$\epsilon \circ m = \epsilon \diam \epsilon;                  
\eqno{\(a6)}$$
\vskip4pt

\smallskip\noindent
a \textit{unit} $\eta: \one \to H$, such
that:
\vskip-4pt
$$m\circ(\id_H\diam\eta) = \id_H = m\circ(\eta\diam\id_H),
\eqno{\(a4-4')}$$
$$\Delta \circ \eta = \eta \diam \eta,
\eqno{\(a7)}$$
$$\epsilon \circ \eta = \id_{\one};
\eqno{\(a8)}$$
\vskip4pt

\smallskip\noindent
an \textit{antipode} $S: H \to H$ and its inverse $\bar 
S: H \to H$, such that:
\vskip-12pt
$$m\circ(S\diam\id_H)\circ\Delta = \eta\circ\epsilon,
\eqno{\(s1)}$$
$$m\circ(\id_H\diam S)\circ\Delta = \eta \circ \epsilon,
\eqno{\(s1')}$$
$$S\circ\bar S =\bar S\circ S=\id_H;
\eqno{\(s2-2')}$$
\vskip4pt

A \textit{braided unimodular Hopf-algebra} in $\C$ is a Hopf algebra $H$ in $\C$, equipped with the
following morphisms in $\C\,$:

\smallskip\noindent
an \textit{integral} $L: \one \to H$ and a \textit{cointegral} $l: H \to \one$, such that:
$$
 (\id_H \diam l) \circ \Delta = \eta \circ l
 \eqno{\(i1)}$$
 $$
 m \circ (L \diam \id_H) = L \circ \epsilon, \eqno{\(i2)}$$
$$
l\circ L=\id_{\one}, \eqno{\(i3)}$$
$$
 l\circ S=l,  \eqno{\(i5)}$$
 $$
S\circ L=L;  \eqno{\(i4)}$$

\smallskip\noindent
A \textit{braided ribbon Hopf-algebra} in $\C$ is a unimodular Hopf algebra $H$ in $\C$, equipped an invertible \textit{ribbon morphism}  $v : H\to H$ and a \textit{copairing} $\sigma : \one \to H \diam 
H$, such that:
\vskip-8pt
$$S \circ v = v \circ S\,,
\eqno{\(r3)}$$
$$\epsilon \circ v = \epsilon\,,
\eqno{\(r4)}$$
$$m \circ (v \diam \id_H) = v \circ m\,,
\eqno{\(r5)}$$
$$\sigma = (v^{-1} \diam (v^{-1} \circ S)) \circ \Delta \circ v \circ \eta,
\eqno{\(r6)}$$
$$(\Delta \diam \id_H) \circ \sigma = (\id_H \diam \id_H \diam 
m) \circ (\id_H \diam \sigma \diam \id_H) \circ \sigma,
\eqno{\(r7)}$$
$$\Delta \circ v^{-1} =(v^{-1} \diam
v^{-1}) \circ \mu \circ  \bar \gamma\circ\Delta,
\eqno{\(r8)}$$
$$
(m \diam m) \circ  (S \diam (\mu \circ \bar\gamma \circ \mu) \diam S) \circ 
(\rho_l \diam \rho_r) = \gamma
\eqno{\(r9)},$$
where $\mu = (m \!\diam m)
\circ (\id_H \diam \sigma \diam \id_H) : H \diam H \to H \diam H$ and $\rho_l = (\id_H \diam m) \circ (\sigma \diam \id_H): H \to H \diam H$ (resp. $\rho_r = (m \diam \id_H) \circ 
 (\id_H \diam \sigma): H \to H \diam H$)  defines a left (resp. right) $H$-comodule structure on $H$.
\end{definition}

\begin{definition}\label{defnHr/def}
The \textit{universal braided ribbon Hopf algebra} $\H^r$ is the braided strict monoidal category freely generated by a ribbon Hopf algebra $H$. 
\end{definition}

In terms of graph diagrams, the elementary morphisms and defining relations of the universal ribbon Hopf algebra $\H^r$ are presented in Table \ref{table-Hr-axioms/fig} (cf. Tables 4.7.12 and 4.7.13 in \cite{BP}), while in Tables \ref{table-Hu-prop/fig} and \ref{table-Hr-prop/fig} we list some important relations which are satisfied in $H^r$. The reader can find their proofs in \cite[Propositions 4.1.4, 4.1.5, 4.1.6, 4.1.9, 4.1.10, 4.2.5, 4.2.7, 4.2.11, 4.2.13]{BP}. 

The existence of a non degenerate form $\lambda$ and coform $\Lambda$, given by  \(f1-2)  in Table \ref{table-Hu-prop/fig}, is well known consequence of the integral axioms \(i1-3)\footnote{In \cite{BD21} the form and the coform are called {\it evaluation} and {\it coevaluation} morphisms.}. This allows to change a graph diagram by isotopy moves which preserve the orientation of its vertices (see \(f3-11) and the duality of the univalent vertices in Table \ref{table-Hu-prop/fig}). Similarly to what is done in \cite{BP}, the use of the isotopy moves in the diagrammatic proofs in the following section  will be implicit.

\begin{definition}
A braided monoidal category $(\C, \diam,\one, \gamma)$ is called symmetric if $\gamma\circ \gamma =\id_{H\diam H}$ and a Hopf algebra $H$ in a symmetric category is called {\it braided cocomutative} if $\gamma\circ\Delta=\Delta$.
\end{definition}

\begin{proposition}\label{Htr/prop}
Any braided cocomutative  Hopf algebra $H$ in a symmetric monoidal category $(\C, \diam,\one, \gamma)$ satisfies the ribbon axioms \(r1-9) with ribbon morphism $v=\id_H$ and trivial copairing $\sigma=\eta\diam\eta$. In particular, any unimodular braided cocomutative  Hopf algebra is a ribbon Hopf algebra with ribbon morphism $v=\id_H$.
\end{proposition}

\begin{proof}
By substituting $v$ with $\id_H$ and $\sigma$ with $\eta\diam\eta$  in  the ribbon axioms \(r1-9) in Table \ref{table-Hr-axioms/fig}, we see that \(r1-6) are trivially satisfied, while by applying axioms \(a4) and \(a7) both sides of \(r7)  reduce to $\eta\diam\eta\diam \eta$ and by applying axioms \(a4-4')  and property \(s6) in Table  \ref{table-Hu-prop/fig}, \(r8) and \(r9) reduce to the defining relations of a braided cocomutative  Hopf algebra: $\gamma\circ\Delta=\Delta$ and $\gamma\circ \gamma =\id_{H\diam H}$. 
\end{proof}

Any classical (non categorical) Hopf algebra over a field $k$ is a Hopf algebra in the (symmetric) tensor category of $k$-modules where the braiding morphism is simply the transposition $\tau$. Such Hopf algebra is called cocommutative if $\tau\circ \Delta =\Delta$. Therefore, Proposition \ref{Htr/prop} implies that any unimodular cocommutative Hopf algebra is a braided cocomutative  ribbon Hopf algebra in the category of $k$-modules. Examples of such algebras are the group algebras $k[G]$, where $G$ is a finite group.

Examples of non braided cocomutative  ribbon Hopf algebras in (non symmetric) braided categories can be obtained from classical non cocommutative unimodular ribbon Hopf algebras. The most famous family of  such algebras come from quantum groups (see for example Chapter 36 of \cite{L93}) and  have been used to define quantum invariants of 3-dimensional manifolds and of 4-dimensional handlebodies (cf. \cite{RT91}, \cite{H96}, \cite{KR95} and \cite{BM02}). Given such algebra,  one can define an appropriate braiding on the category of  its two-sided modules and modify its coproduct in a way to obtain a ribbon Hopf algebra in this last category. 

\smallskip
Now we can address the problem of the independence of the set axioms of $H^r$. 

\begin{proposition}\label{indep-axioms/prop}
(a) The integral axioms \(i1-4) are independent from the rest of the axioms of $\H^r$;

(b) The ribbon axioms $(r8-9)$ are independent from the rest of the axioms  of $\H^r$ .
\end{proposition}

\begin{proof}
(a) follows from a theorem of   Sweedler \cite{S69}, according to which a Hopf algebra over a field $k$ has an integral if and only if it is finite dimensional. Therefore, it is enough to find a  categorical infinite dimensional Hopf algebra which satisfies the ribbon axioms. One example\footnote{One can take as well  the universal enveloping algebra of any Lie algebra.}  is the  Hopf algebra $H_1=k[x]$ with $\Delta(x)=x\otimes 1+1\otimes x$, $\epsilon(x)=0$ and $S(x)=-x$.  $H_1$ is an infinite dimensional  braided cocomutative  Hopf algebra in the symmetric category of  $k$-modules and according to Proposition \ref{Htr/prop}, it satisfies  the ribbon axioms \(r1-9) with  $v=\id_H$. 

To show (b), suppose that $(\C,\, \gamma, \,H,\, v, \,\sigma)$ is ribbon Hopf algebra with ribbon morphism $v$ and copairing $\sigma$ in a non symmetric braided category $\C$; in particular $\gamma_{H,H}\circ\gamma_{H,H}\neq \id_{H\diam H}$. Consider now $(\C,\, \gamma, \,H,\, v'=\id_H, \,\sigma=\eta\diam\eta)$, i.e. the same Hopf algebra in the same category, but with a different (trivial) choice of ribbon morphism and copairing. All axioms are still satisfied with the exception of \(r8-9), since the choice of the ribbon morphism is not compatible any more with the braiding and the comultiplication. 

Therefore, in order to prove that axioms \(r8-9) are independent from the rest of the axioms, it is enough to present an example of a non symmetric braided category and a unimodular ribbon Hopf algebra in it which satisfies all axioms in Definition \ref{cat-hopf-algebra/def}. This is done  in  \cite{BD21}, where it is shown that the transmutation of the small quantum group $u_qsl_2$ is a unimodular ribbon Hopf algebra  in the category $u_qsl_2$-mod  of finite-dimensional left $u_qsl_2$-modules and this last category is non symmetric.
\end{proof}

\medskip\smallskip
\noindent{\bf Remarks}
\begin{enumerate}
\item Relation \(r6) in Definition \ref{cat-hopf-algebra/def} can actually be seen as the definition of $\sigma$ in terms of the ribbon morphism $v$ and this is how it is presented in Table \ref{table-Hr-axioms/fig}. Nevertheless,  the set of axioms can be changed by removing \(r6) and introducing property \(p2) in Table \ref{table-Hr-prop/fig} as an axiom; then \(r6) will follow from \(s3), \(p2) and \(r8).

\item Our definition of categorical ribbon Hopf algebra includes the condition that the algebra is unimodular (axioms \(i1-3)). This is not a standard choice of terminology and it was done in an attempt to find a single name which illustrates the most important characteristic of the algebra. 

\item As it was seen in Proposition \ref{Htr/prop}, the axioms of $\H^r$ are compatible with trivial copairing and therefore do not imply that the copairing $\sigma$ is non degenerate. This is the main difference between $\H^r$ and the category $\Hbb^r$, where the non degeneracy of $\sigma$ will be imposed (see Section \ref{Hbb/sec}).

\item The result in  \cite{BD21}, used in the proof of Proposition \ref{indep-axioms/prop} (b), is actually more general. In  \cite{BD21} it shown that if $H$ is any unimodular ribbon (non categorical) Hopf algebra, its transmutation, as defined by Majid in \cite{Ma91}, is a unimodular ribbon Hopf algebra  in the category  of finite-dimensional left $H$-modules in the sense of Definition  \ref{cat-hopf-algebra/def}. In particular, the transmutation of the small quantum group $u_q \mathfrak{g}$ for any simple complex Lie algebra $\mathfrak{g}$ is an example such algebra.

\begin{table}[htb]
\centering \fig{}{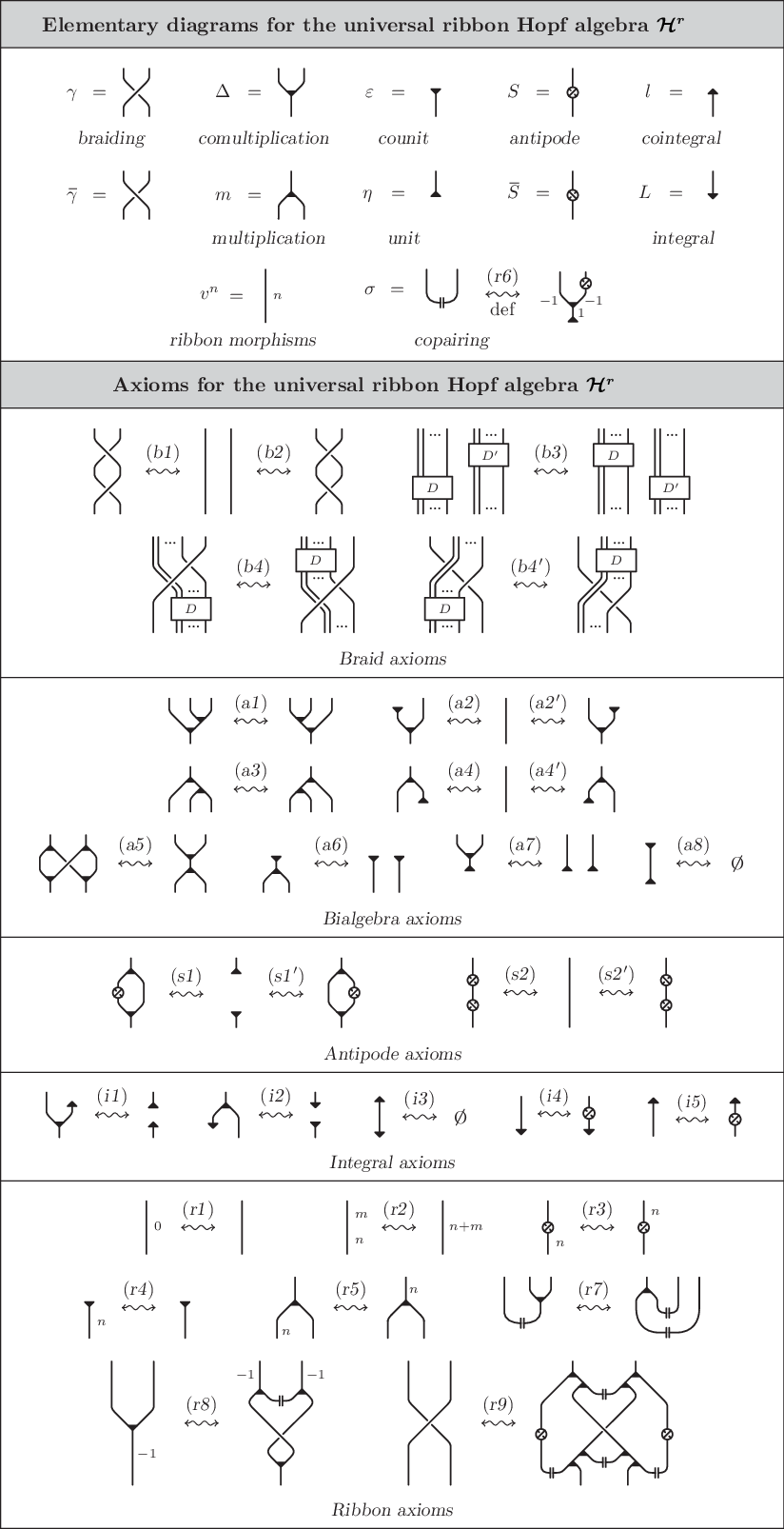}
\caption{}
\label{table-Hr-axioms/fig}
\end{table}

\begin{table}[htb]
\centering \fig{}{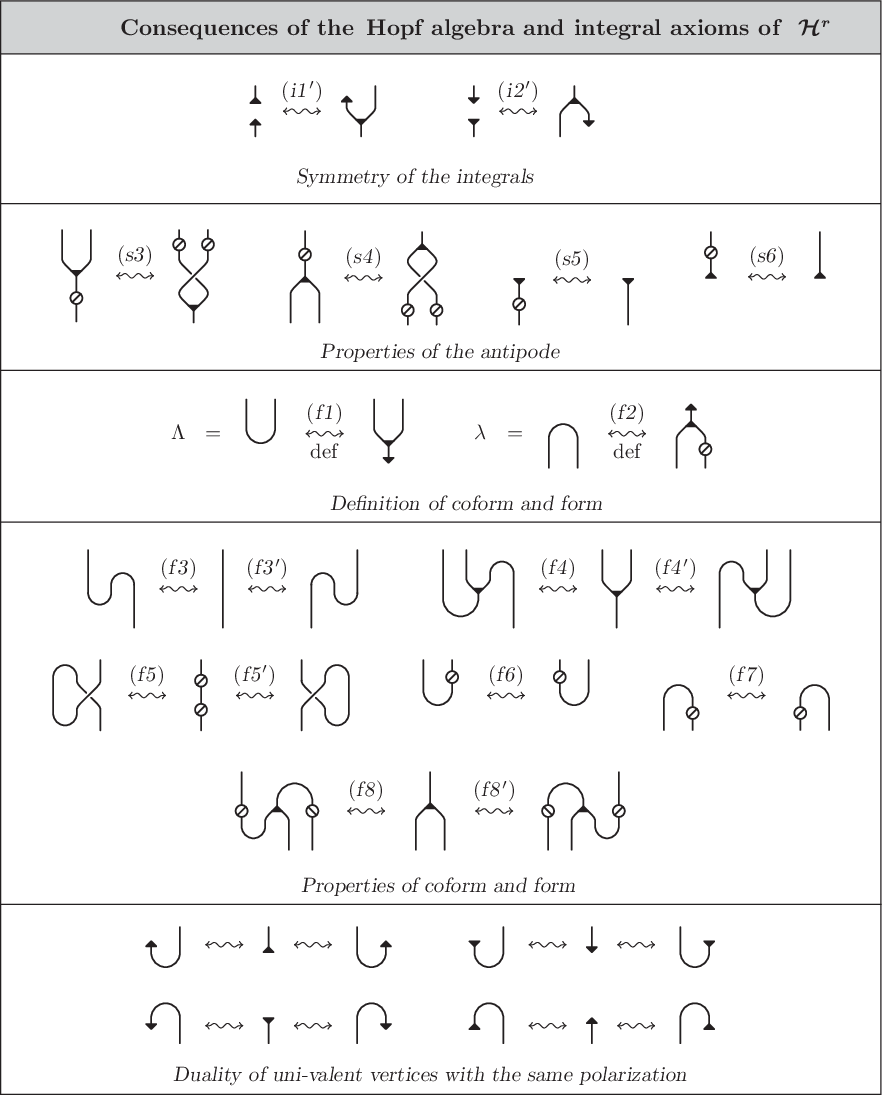}
\caption{}
\label{table-Hu-prop/fig}
\end{table}

\begin{table}[htb]
\centering \fig{}{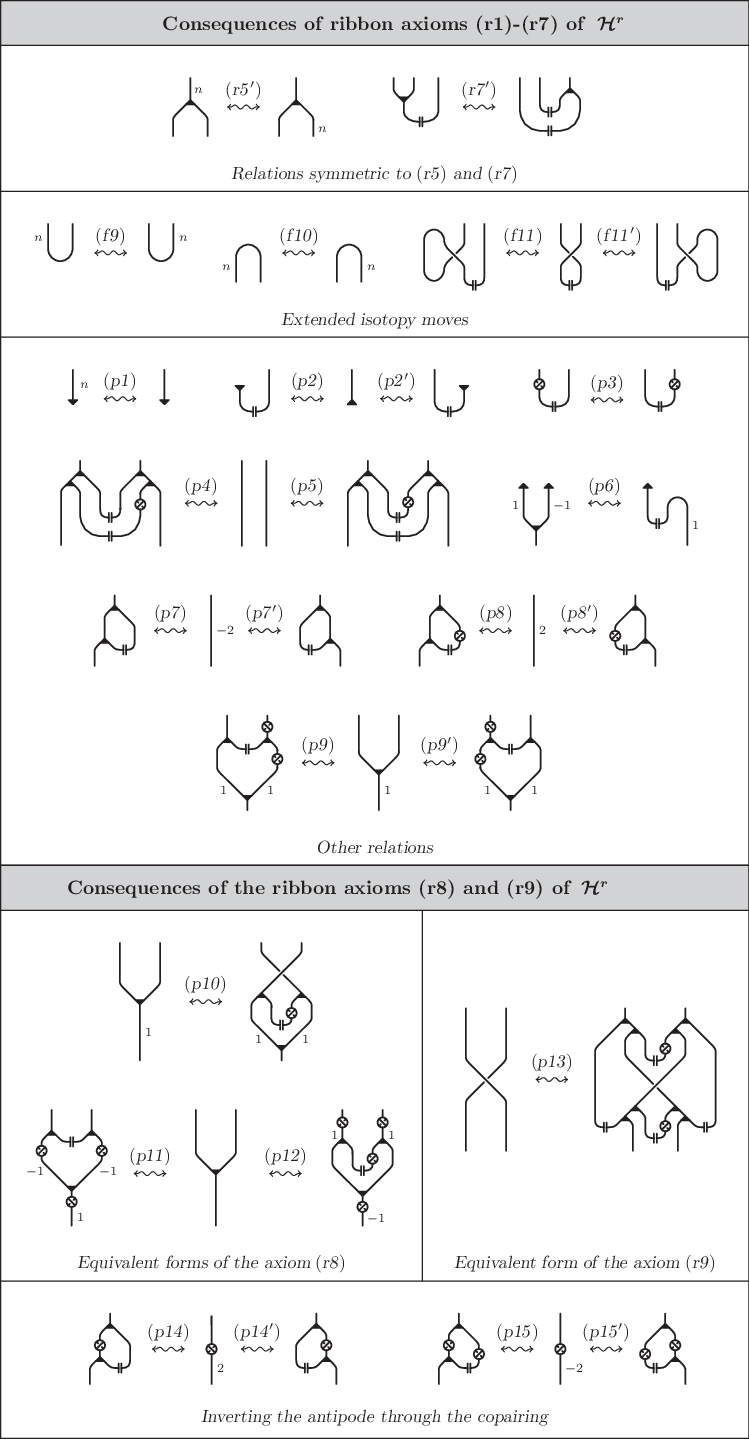}
\caption{}
\label{table-Hr-prop/fig}
\vskip-3pt
\end{table}

\clearpage

 \end{enumerate}


\subsection{The category  $\Chb^{3+1}$ }\label{Cob4/sec}

We start by presenting a brief review of  the definition of the category of relative 4-dimensional 2-handlebody cobordisms $\Chb^{3+1}$ with a single 0-handle, and its description in terms of Kirby tangles. Full details  and references to the broad literature on this argument   \cite{Ki89, GS99, Ke99, KL01} can be found in  Chapter 2 of \cite{BP}.

\medskip
The set of objects in $\Chb^{3+1}$ is $\{M_n,\, \iota_n\}_{n\in \mathbf N}$, where $M_n$ is  a standard 3-dimensional relative 1-handlebody with a single 0-handle and $n$ 1-handles, and $\iota_n: D^2\to \partial H^0$ is an embedding of the 2-dimensional disk in the boundary of the 0-handle. A morphism $W: M_n\to M_m$ in $\Chb^{3+1}$ is a relative 4-dimensional 2-handlebody build on  the connected 3-dimensional 1-handlebody $X(M_n, M_m)=(M_n\sqcup M_m)\cup_{\iota_n\sqcup\iota_m}\,D^1\times D^2$, obtained by attaching a single 1-handle  between the 0-handles of $M_n$ and $M_m$ (see Figure \ref{cobordism01/fig} for an example, where all the horizontal segments represent copies of $B^2$). Such 4-dimensional 2-handlebody is considered up to {\it 2-equivalence}, i.e. up to changing  the attaching maps of its 1- and 2-handles by isotopy and creation/cancellation of 1/2 handles pairs.

 \begin{figure}[h]
\centering \fig{}{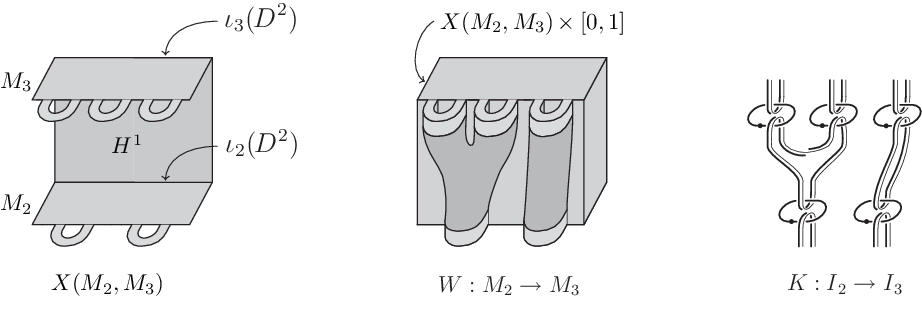}
\caption{Example of a morphisms in $\Chb^{3+1}$ and its representation with a Kirby tangle.}
\label{cobordism01/fig}
\vskip-6pt
\end{figure}

\medskip
The morphisms of $\Chb^{3+1}$ are represented diagrammatically by admissible  Kirby tangles and this allows  to identify $\Chb^{3+1}$  with the category of such tangles \cite[Proposition 2.3]{BP}.

\smallskip
In particular, the category of admissible Kirby tangles has as set of objects sequences  $I_n = $ $((a'_{1},a''_{1}), \dots, (a'_{n},a''_{n}))$ of pairs of intervals in $E = [0,1]^2$ associated to each $M_n,\,n\in \mathbf N$. Then an \textit{admissible Kirby tangle} $K: I_n\to I_m$ is a tangle with two types of components:
\begin{itemize}
\item[-]
Dotted unknots spanning disjoint flat disks in $\text{Int}\, E \times \left]0,1\right[$. Such unknots represent the 1-handles of the 4-dimensional 2-handlebody whose attaching regions (disjoint pairs of 3-balls)  are thought to be squeezed onto the spanning disks. Observe that the spanning discs won't be explicitly drawn in the plane diagrams but will be assumed to project bijectively onto l'interior of the unknots.
\item[-]
 Framed curves  regularly embedded in $\text{Int}\, E \times \left[0,1\right]$ and transversal to the spanning disks of the unknots; such curves represent the attaching maps of the 2-handles and will be drawn  as ribbons in which  the base curve is thicker then the parallel framing curve. Each open curve  joins a pair of intervals $(a'_{k} \times \{0\},a''_{k} \times \{0\})$ for some $(a'_{k},a''_{k}) \in I_n$ or $(a'_{k} \times \{1\},a''_{k} \times\{1\})$ for some $(a'_{k},a''_{k}) \in I_m$ and the base curve always ends in the left end-points of the intervals.
\end{itemize}

 The {\it composition} of two tangles $K_1: I_{n} \to I_{m}$ and $K_2: I_{m} \to I_{p}$  is given by translating $K_2$ on the top of $K_1$, glueing the two tangles along $I_{m}$ and then rescaling.
 Two  2-handlebodies  are \textit{2-equivalent} if and only if the corresponding Kirby tangles are related by isotopy and the moves of Figure \ref{kirby-tang-moves/fig}.
\begin{figure}[htb]
\centering \fig{}{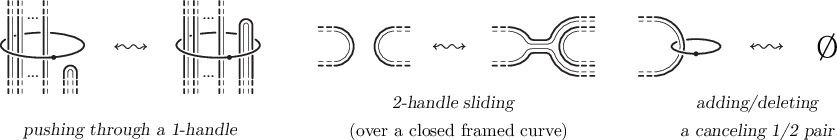}
\caption{Equivalence moves for  Kirby tangles}
\label{kirby-tang-moves/fig}
\vskip-6pt
\end{figure}

The category of admissible Kirby tangles has a strict monoidal structure, whose product, once again denoted by $\diam$ , is given by $I_n\diam I_m=I_{n+m}$ on the objects, while on the morphisms $K\diam K'$ is obtained by translating $K'$ in the space on the right of $K$ and rescaling. The unit of the product is the empty tangle.

\smallskip
The following is Theorem 4.7.5 in \cite{BP}.

\begin{theorem}{\rm\bf \cite{BP}} The universal algebraic category $\H^r $  is equivalent to the category of relative cobordisms of 4-dimensional 2-handlebodies $\Chb^{3+1}$. The equivalence functor $\Phi :\H^r\to \Chb^{3+1}$ sends the elementary morphisms,  the form, the coform and the copairing of $\H^r$ to the corresponding Kirby tangles in Figure \ref{phi/fig}.
\end{theorem}

\begin{figure}[htb]
\centering \fig{}{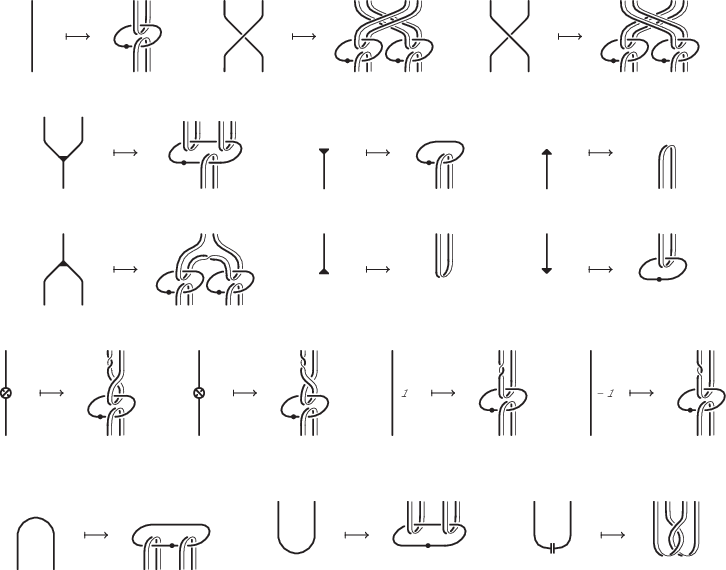}
\caption{$\Phi\!:\H^r \to \Chb^{3+1}$}
\label{phi/fig}
\end{figure}

\medskip
{\bf Remarks}

\begin{enumerate}
\item
As it is seen in Figure \ref{phi/fig}, the images under $\Phi$  of the cointegral $\lambda$ and of  the unit $\eta$ have the same topological structure in $\Chb^{3+1}$ with the source and the target exchanged. The same is true for the images of the integral  $\Lambda$  and of the counit $\epsilon$. Moreover, in $\Chb^{3+1}$ both the integral axioms \(i1-3) and the defining axioms of the unit and the counit \(a2-2') and \(a4-4') reduce to a cancellation of 1/2 handle pairs. Nevertheless, as we saw in Proposition \ref{indep-axioms/prop} (a), on algebra level the integral axioms do not follow from the rest of the axioms.

\item \label{nondegeneracy/rem}
One can define a pairing $\bar\sigma =(\lambda\diam\lambda)\circ (\id_H\diam \sigma\diam S): \H^r\diamond \H^r\to \one$ (Figure \ref{defn-bar-sigma/fig} (a)), but as it is shown in Figure \ref{defn-bar-sigma/fig} (b),    $\Phi((\bar\sigma\diamond\id_H)\circ(\id_H\diamond\sigma))\neq\id_H$. Actually,  substituting $\Phi((\bar\sigma\diamond\id_H)\circ(\id_H\diamond\sigma))$ with the identity would correspond to substituting an undotted component with a dotted one (1/2-handle trading). Such move would change the interior of the corresponding 4-dimensional handlebody, leaving unchanged its boundary. Therefore, the non degeneracy of $\sigma$ doesn't hold in $\H^r$, but should be imposed  in the definition of the boundary category $\Hbb^r$,  equivalent to the category of 3-dimensional cobordisms $\Cob^{2+1}$.

\begin{figure}[htb]
\centering \fig{}{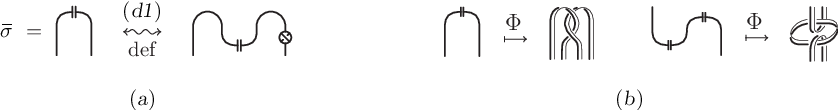}
\caption{Definition of $\overline{\sigma}$ and its image under $\Phi$.}
\label{defn-bar-sigma/fig}
\vskip-6pt
\end{figure}

\item Observing the images of the ribbon element and the copairing under the equivalence functor $\Phi$ in Figure \ref{phi/fig}, we deduce that invariants coming from braided cocomutative  Hopf  algebras, where the ribbon morphism and the copairing are trivial, can not detect changes in the framing or in the linking of the attaching maps of the two handles. It can easily be seen that such invariants will depend only on the 2-dimensional spine of the handlebody up to  2-deformations. For example, group algebras of finite groups will provide  well-known homotopy invariants of 4-dimensional 2-handlebodies.

\end{enumerate}

\subsection{Adjoint morphisms in $\H^r$}
\label{adjoint/sec}
 
The adjoint morphisms $\alpha_n: H\diam H^{\diam n}\to H^{\diam n}$ have been used in Chapter 4.4 of \cite{BP} to give an alternative formulation of the ribbon axioms \(r8) and \(r9).   We present here the definition and list the basic properties of such morphisms.
 
 \begin{definition}\label{alpha/def}
Define the (left) adjoint morphism $\alpha_n: H\diam H^{\diam n}\to H^{\diam n}$ inductively by the following identities (see Table \ref{table-adjointdefn/fig}):
\vskip-6pt
$$\arraycolsep0pt
\alpha_1 =  m \circ (m \diam S) \circ  (\id_H \diam \gamma_{H,H}) \circ (\Delta \diam \id_H) 
 \eqno{\(q1)}
$$
$$
\alpha_{n+1} = (\alpha_n \diam 
 \alpha_1) \circ (\id_{H} \diam \gamma_{H, H^{\diam n}} \diam 
 \id_{H}) \circ (\Delta\diam \id_{H^{\diam (n+1)}}).
 \eqno{\(q2)}
$$
\end{definition}

\begin{table}[hbt]
\centering \fig{}{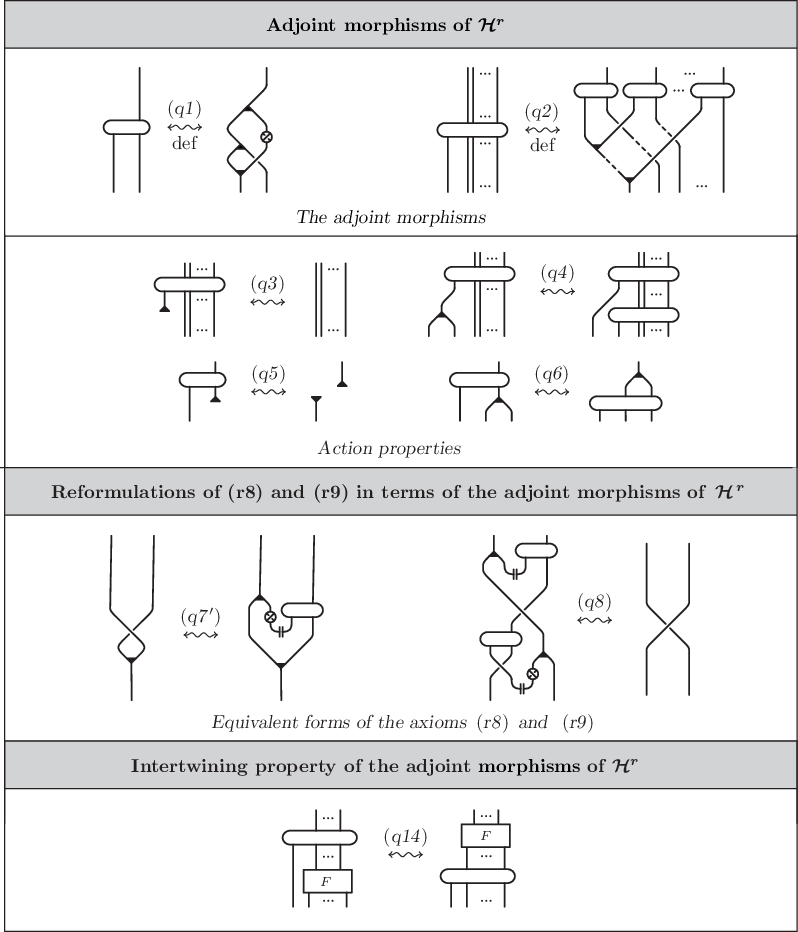}
\caption{}
\label{table-adjointdefn/fig}
\end{table}

In Table \ref{table-adjointdefn/fig}  are presented the properties of the adjoint morphisms which are relevant to the present work (see  Tables 4.4.1 and 4.4.5 in \cite{BP} for the complete list). Observe that, analogously to the case of classical Hopf algebras, properties \(q3-6) imply that the adjoint morphism defines a left action of $H$ on itself and intertwines with the multiplication and the unit morphisms.  The proofs of those relations are straightforward and do not make use of the ribbon axioms. Moreover, the ribbon axioms \(r8) and \(r9) can be reformulated in terms of the adjoint morphisms in the form \(q7') and \(q8) presented in Table  \ref{table-adjointdefn/fig}  \cite[Proposition 4.4.5]{BP}. 

The last property \(q14) in Table \ref{table-adjointdefn/fig} states that not just the multiplication and the unit, but any  morphism  $F: H^{\diam n}\to H^{\diam m},\, n, m \geq 0$ in $H^r$ intertwines with the adjoint morphism \cite[Lemma 4.4.10]{BP}:
$$
F\circ\alpha_n=\alpha_m\circ(\id_H\diam F) \eqno{\(q14)}
$$
 This is highly non trivial fact.  In the classical case of  symmetric tensor category, the adjoint action is known to intertwine with the comultiplication and the antipode only when the Hopf algebra is cocommutative \cite[Lemma 5.7.2]{M93}. A ribbon Hopf algebra with non-trivial braiding is not cocommutative and if the proof of Proposition 4.4.5 in \cite{BP} works, it is only because the ribbon axioms \(r8) and \(r9) are the right ones to make the comultiplication, the antipode and the braiding morphisms intertwine with the adjoint action. A legitimate question is if, modulo the rest of the axioms of $\H^r$, the ribbon axioms \(r8) and \(r9) are equivalent to the condition that the comultiplication, the antipode and the braiding intertwine with the adjoint morphism. In the case of the quotient category $\Hbb^r$, Proposition \ref{new-axiomsHbb/thm} bellow makes a step in this direction.

The topological meaning of the intertwining property \(q14) can be understood by looking at the image of the adjoint morphism under the functor $\Phi:H^r\to Cob^{3+1}$. Such image consists in a single dotted component which embraces some open framed components (Figure \ref{adjoint-phi/fig}). In other words, the adjoint morphism represents an 1-handle along which pass the attaching maps of some 2-handles. Therefore the intertwining of a given morphism with the adjoint one is the algebraic analog of pushing part of the Kirby diagram through an 1-handle. 

 \begin{figure}[h]
\centering \fig{}{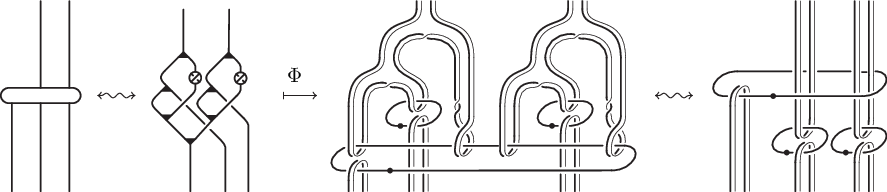}
\caption{The image of the adjoint morphism under $\Phi$.}
\label{adjoint-phi/fig}
\vskip-6pt
\end{figure}

\subsection{The universal boundary ribbon Hopf algebra $\overline{\overline{\H}}$$^r$}\label{Hbb/sec}

The category of  3-dimensional relative cobordisms $\Cob^{2+1}$ can be seen as the boundary of $\Cob^{3+1}$ and, following \cite{KL01}, we define it as the quotient category of $\Chb^{3+1}$ modulo the well known Kirby moves \cite{Ki78} relating any
two 4-dimensional 2-handlebodies with diffeomorphic boundaries.  In particular, the objects of  $\Cob^{2+1}$  are those of $\Chb^{3+1}$, while the morphisms of $\Cob^{2+1}$ are equivalence classes of morphisms of $\Chb^{3+1}$ under the relations generated by 1/2-handle trading and blowing down/up (Figure \ref{boundary01/fig}).

\begin{figure}[htb]
\centering \fig{}{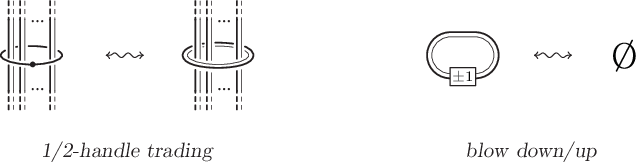}
\caption{Kirby calculus moves}
\label{boundary01/fig}
\end{figure}

As it was observed in Remark 2 on p. \pageref{nondegeneracy/rem}, in $\H^r$ the invariance under 1/2-handle trading corresponds to the non-degeneracy of the copairing. Moreover, in order to obtain the such nondegeneracy, it is enough to require that the  integral is dual to the cointegral with respect to the copairing \cite[Proposition 5.4.2]{BP}. 
This motivates the following definition.

\begin{definition} The \textit{universal selfdual ribbon Hopf algebra} $\Hbb^r$ is the
quotient category of $\H^r$ modulo the relations (see Table \ref{table-Hrb/fig}):
 $$
 (\lambda\diam \id_H)\circ\sigma=\Lambda,\,\eqno{\(d)}
 $$
 $$
 \lambda\circ v\circ\eta=\one,\,\eqno{\(n)}
 $$
\end{definition}

The axioms and the properties of $\Hbb^r$ are presented in Table \ref{table-Hrb/fig}. In particular, properties \(d2-2') and \(d3-3') in Table \ref{table-Hrb/fig} assert  the existence of a non degenerate pairing $\bar\sigma$ which defines an isomorphism between the algebra and its dual. 

\begin{table}[htb]
\centering \fig{}{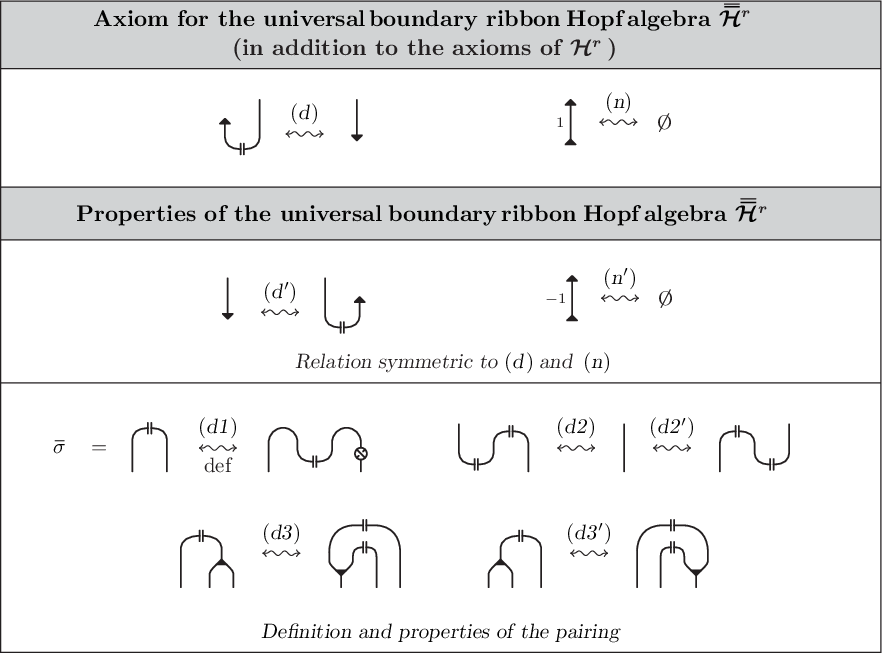}
\caption{}
\label{table-Hrb/fig}
\end{table}

The following theorem   \cite[Theorem 5.5.4]{BP} solves the problem posed by Kerler.

\begin{theorem} {\rm \bf\cite{BP}} \label{alg3cob/thm} The functor $\Phi$ induces an equivalence of categories $\bbPhi:\Hbb^{r} \to \Cob^{2+1}$. 
\end{theorem}

Theorem \ref{alg3cob/thm} implies  that any $2+1$-dimensional TQFT is generated by a (categorical) selfdual ribbon Hopf algebra and also that any TQFT of 4-dimensional 2-handlebodies, generated  by  such algebra won't detect more then the boundaries of those handlebodies and, in the case in which the normalization condition \(n) is not satisfied, their signature. 


\subsection{The universal category $\alg$}

We now introduce the category $\alg$ as it is defined in \cite{AS}. 

\begin{definition}\label{alg/defn}
$\alg$ is the universal braided strict monoidal category, generated by a Hopf algebra $H$, equipped with \textit{ribbon elements}  $w_+: \one\to H$, $w_- : \one\to H$,  \textit{copairing} $\sigma : \one \to H \diam H$ and  \textit{pairing} $\bar\sigma :  H \diam H \to\one$, such that:
\vskip-8pt
$$m \circ (w_+ \diam \id_H) = m \circ ( \id_H \diam  w_+)\,,
\eqno{({\mathtt h1})}$$
$$m\circ(w_+\diam w_-)=\eta\,,
\eqno{({\mathtt h2})}$$
$$\epsilon \circ w_+ = \one,
\eqno{({\mathtt h3})}$$
$$S \circ w_+ = w_+,
\eqno{({\mathtt h4})}$$
$$\Delta\circ  w_+=(m\diam m)\circ (w_+ \diam \sigma\diam w_+),
\eqno{({\mathtt h5})}$$
$$(\epsilon\diam \id_H)\circ \sigma=\eta=(\id_H\diam\epsilon)\circ \sigma,
\eqno{({\mathtt h6-6'})}$$
$$( \id_H\diam \Delta) \circ \sigma = (m\diam\id_H \diam \id_H) \circ (\id_H \diam \sigma \diam \id_H) \circ \sigma,
\eqno{({\mathtt h7})}$$
$$(\Delta \diam \id_H) \circ \sigma = (\id_H \diam \id_H \diam 
m) \circ (\id_H \diam \sigma \diam \id_H) \circ \sigma,
\eqno{({\mathtt h7'})}$$
$$(\id_H\diam \bar\sigma)\circ (\sigma\diam\id_H)=\id_H=(\bar\sigma\diam\id_H)\circ(\id_H\diam \sigma),
\eqno{({\mathtt h8-8'})}$$
$$m\circ(w_+\diam w_+)=m\circ \sigma,
\eqno{({\mathtt h9})}$$
$$(\alpha_1\diam\id_H)\circ (\id_H\diam \gamma)\circ(\Delta\diam\id_H)=\overline{\gamma}\circ (\id_H\diam \alpha_1)\circ(\Delta \diam \id_H),
\eqno{({\mathtt h10})}$$
$$d\circ(m\diam\id_H)\circ(w_+\diam w_+\diam w_+)=\one,
\eqno{({\mathtt h11})}$$
$$ S^2=(\id_H\diam \bar\sigma)\circ(\gamma\circ\id_H)\circ(\id_H\diam \sigma),
\eqno{({\mathtt h12})}$$
\end{definition}

The elementary morphisms and defining relations of $\alg$ are presented in Table \ref{table-Alg-axioms/fig}, where the notations adopted for all elementary morphisms with the exception of the ribbon elements, is the same as the one used for the analogous morphisms of $H^r$  in Tables \ref{table-Hr-axioms/fig} and  \ref{table-Hrb/fig}. This won't cause a confusion since, as we will see bellow, the functor from $\alg$ to $\Hbb^r$ sends these morphisms  to the corresponding ones in $\Hbb^r$. On the other hand, for the ribbon element of $\alg$ and its inverse  (which do not appear as elementary morphism of $\Hbb^r$) we have kept  the notation used in \cite{AS}.  

\smallskip
We observe that $\Hbb^r$  and $\alg$ present different choices for the set of ribbon axioms. Moreover, in the definition of $\alg$ the integral morphisms  and the normalization axiom \(n) are missing. 

\begin{table}[H]
\centering \fig{}{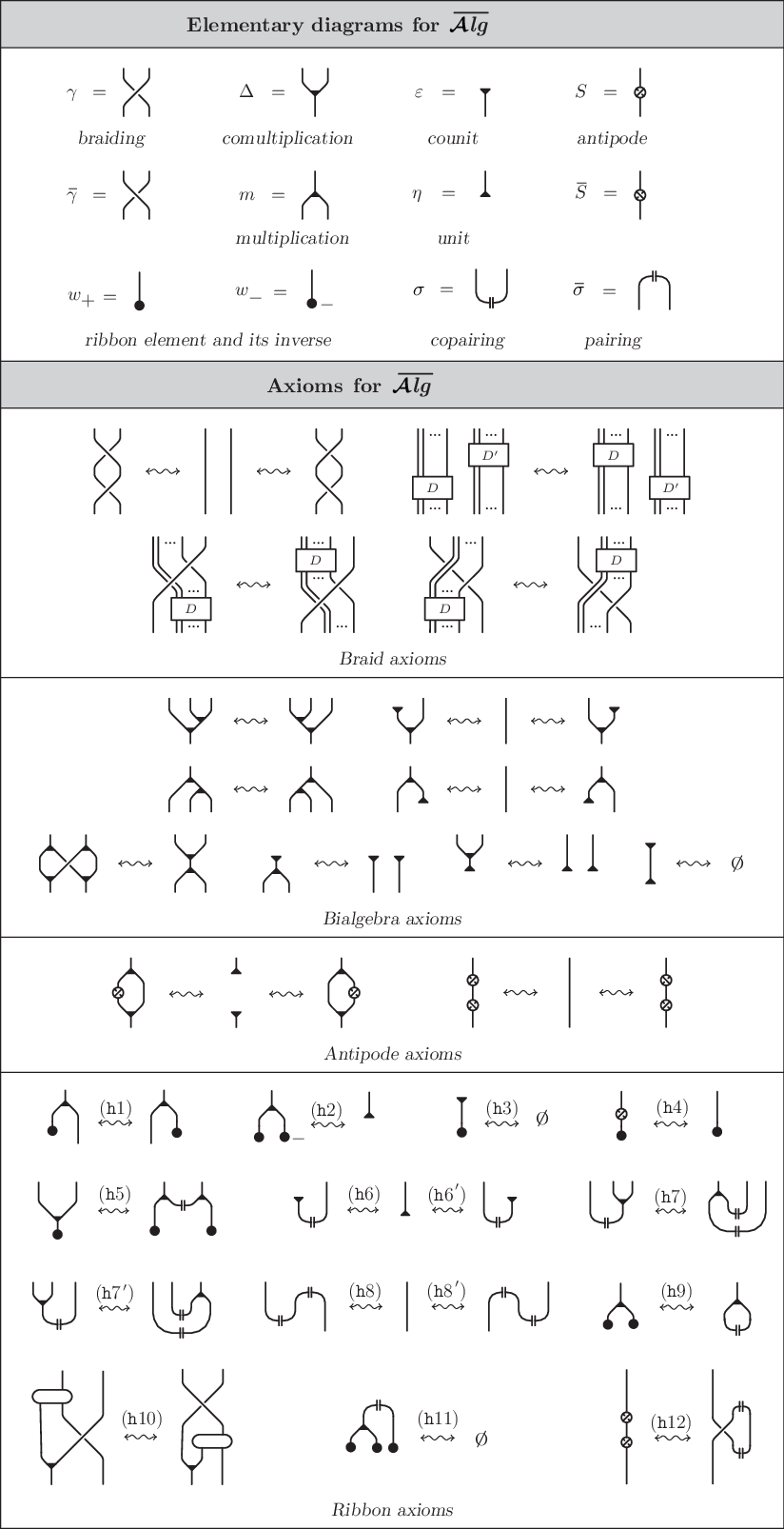}
\caption{}
\label{table-Alg-axioms/fig}
\end{table}


\section{The functor $\Gamma:\alg\to \overline{\overline{\H}}$$^r$}\label{Gamma/sec}

In this section we define the functor from $\alg$ to $\Hbb^r$ and present an alternative set of axioms for $\Hbb^r$.

\medskip
The proofs in the section will consist in showing that some morphisms in the universal algebra $\Hbb^r$ are equivalent, meaning that the graph diagram of one of them can be obtained from the graph diagram of the other by applying a sequence of the algebra's axioms and  properties. We will outline the main steps in this procedure by drawing in sequence some intermediate diagrams, and for each step we will indicate in the corresponding order, the main moves from Tables \ref{table-Hr-axioms/fig}, \ref{table-Hu-prop/fig}, \ref{table-Hr-prop/fig}, \ref{table-adjointdefn/fig} and  \ref{table-Hrb/fig} needed to transform the diagram on the left into the one on the right.  Notice, that the moves represent equivalences of diagrams and we will use the same notation for them and their inverses. On the other hand, we recall that the use of the isotopy moves \(f3-11) and of the duality of the univalent vertices in Table \ref{table-Hu-prop/fig} will be implicit.

\medskip
In order to prove that there exist a functor $\alg\to\Hbb^r$, we need the following lemma, which shows that  relations  analogous to  $({\mathtt h5})$ and $({\mathtt h10})$ in Definition \ref{alg/defn} of $\alg$  are satisfied in $\H^r$, i.e the proof of these relations doesn't require that the copairing is non degenerate.

\begin{lemma} \label{gamma1/lemma} The following relations are satisfied in $\H^r$:
$$\Delta \circ v^{-1} \circ \eta=(v^{-1} \diam v^{-1}) \circ\sigma \, \eqno{\(h5)}$$
$$\alpha_1\circ (\id_H\diam \gamma)\circ(\Delta\diam\id_H)=\overline{\gamma}\circ (\id_H\diam \alpha_1)\circ(\Delta \diam \id_H)\, \eqno{\(h10)}$$
\end{lemma}

\begin{proof} The proofs are presented in Figures \ref{proof-h5/fig} and \ref{proof-h10/fig}. Observe that the proof of \(h10) uses relations \(q7') and \(q8) in Table \ref{table-adjointdefn/fig}, equivalent to axioms \(r8) and \(r9).
\end{proof}

\begin{figure}[htb]
\centering \fig{}{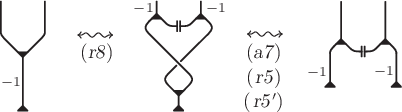}
\caption{Proof of {\sl (h5)}.}
\label{proof-h5/fig}
\vskip-6pt
\end{figure}

\begin{figure}[htt]
\centering \fig{}{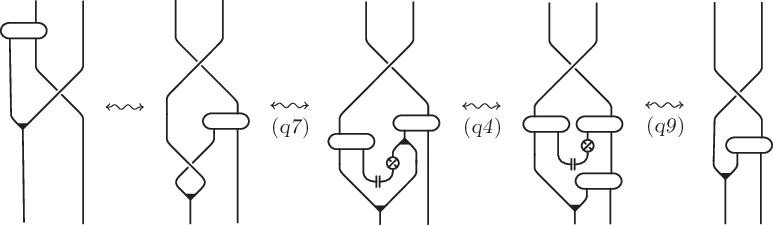}
\caption{Proof of  {\sl (h10)}.}
\label{proof-h10/fig}
\vskip-6pt
\end{figure}


\begin{theorem} \label{functor-Hab/thm}
There exists a braided monoidal  functor $\Gamma: \alg\to \Hbb^r$ which preserves the Hopf algebra structure and sends the pairing and the copairing in $\alg$ to the corresponding ones in $\H^r$. Moreover, (see Figure \ref{defn-gamma/fig}):
$$ \Gamma(w_+)=v^{-1}\circ\eta,\qquad \Gamma(w_-)=v\circ\eta. $$
\end{theorem}

\begin{figure}[htb]
\centering \fig{}{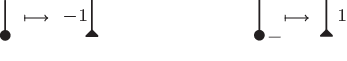}
\caption{Images under  $\Gamma: \alg\to$$\overline{\overline{\H}}$$^r$ of the ribbon element and its inverse. }
\label{defn-gamma/fig}
\vskip-6pt
\end{figure}

\begin{proof}
It is enough to show that the defining relations of $\alg$ in Table \ref{table-Alg-axioms/fig} are satisfied in the image of $\Gamma$. The braid, the bialgebra and the antipode axioms  of $\alg$ are the same as the ones of $\Hbb^r$, so we only need to check the ribbon axioms. Most of them coincide or follow directly from axioms or properties of $\H^r$:

\smallskip
\(h1) follows from \(r5-5') in Tables \ref{table-Hr-axioms/fig} and  \ref{table-Hr-prop/fig};

\(h2)  follows from \(r1-2-5-5') in Tables \ref{table-Hr-axioms/fig} and \ref{table-Hr-prop/fig};

\(h3) follows from \(a8) and \(r4) in Table \ref{table-Hr-axioms/fig};

\(h4) follows from \(r3) and \(s6) in Tables \ref{table-Hr-axioms/fig} and \ref{table-Hu-prop/fig};

\(h5) and \(h10) follow from Proposition \ref{gamma1/lemma};

\(h6-6') coincide with \(p2-2');

\(h7-7') coincide with \(r7-7') in Tables \ref{table-Hr-axioms/fig} and \ref{table-Hr-prop/fig};

\(h8-8') coincide with \(d2-2') in \ref{table-Hrb/fig};

\smallskip
\noindent The proofs of the remaining relations \(h9), \(h11) and \(h12) are presented in Figures  \ref{proof-h9/fig},  \ref{proof-h11/fig}  and \ref{proof-h12/fig}. 
\end{proof}

\begin{figure}[H]
\centering \fig{}{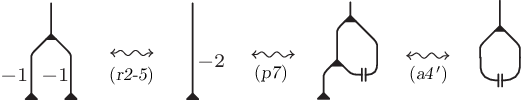}
\caption{Proof of {\sl (h9)}.}
\label{proof-h9/fig}
\vskip-6pt
\end{figure}

\begin{figure}[H]
\centering \fig{}{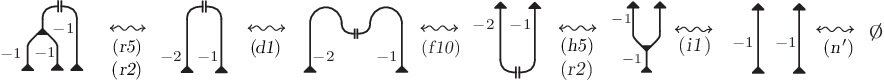}
\caption{Proof of relation {\sl (h11)}.}
\label{proof-h11/fig}
\vskip-6pt
\end{figure}

\begin{figure}[H]
\centering \fig{}{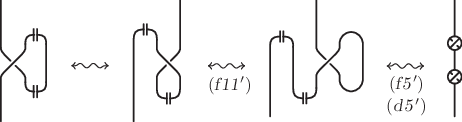}
\caption{Proof of relation {\sl (h12)}.}
\label{proof-h12/fig}
\vskip-6pt
\end{figure}

We observe that the requirement that the Hopf algebra is selfdual is a strong one and it is legitimate to ask if in $\Hbb^r$, axioms \(r8-9) are still independent from the rest or if they can be presented in simpler form.  It seems to us that the independence still holds, but the next proposition, the second part of which was communicated to the author by Stamatova, shows that indeed, in $\Hbb^r$ \(r9) is equivalent to \(h10) and \(r8) is equivalent to the requirement that the copairing intertwines with the adjoint action  (see Figure \ref{relations-qh/fig}). 


\begin{figure}[htb]
\centering \fig{}{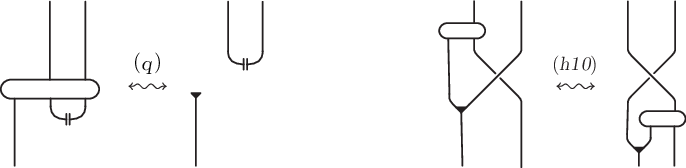}
\caption{Relations equivalent to $(r8)$ and $(r9)$ in $\overline{\overline{\H}}$$^r$.}
\label{relations-qh/fig}
\end{figure}

\begin{proposition}\label{new-axiomsHbb/thm} In $\Hbb^{r}$, modulo the rest of its axioms
\begin{itemize} 
\item[(a)] axiom  \(r8) is equivalent to relation \(q);
  
\item[(b)] {\rm [Stamatova]}  axiom \(r9) is equivalent to relation \(h10).
  \end{itemize}
  
  \medskip
Therefore $\Hbb^{r}$ is equivalent to the universal algebraic category, generated by the elementary morphisms and axioms in Tables \ref{table-Hr-axioms/fig} and \ref{table-Hrb/fig}  where axioms \(r8) and \(r9) are replaced by relations \(q) and \(h10).
\end{proposition}

\begin{proof}
We remind that according to Lemma \ref{gamma1/lemma} and property \(q14) in Table \ref{table-adjointdefn/fig}, both relations \(q) and \(h12) hold in $\Hbb^r$. Moreover, modulo the other axioms of $\Hbb^{r}$, \(r8) is equivalent to \(q7') and \(r9) is equivalent to \(q8)  in Table \ref{table-adjointdefn/fig}. Therefore, it is enough to show that \(q7') follows from \(q) and the rest of the axioms of $\Hbb^r$, excluded \(r8), and  that \(q8) follows from \(h12) and the rest of the axioms of $\Hbb^r$, excluded \(r9). This is done in Figures \ref{proof-ir8/fig} and \ref{proof-hr9/fig}.
\end{proof}

\begin{figure}[htb]
\centering \fig{}{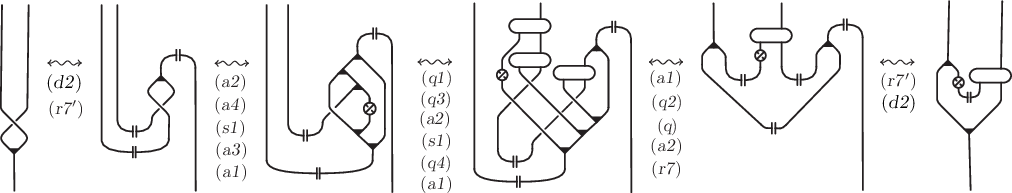}
\caption{Proof that in $\overline{\overline{\H}}$$^r$ $(q7')$ follows from $(q)$.}
\label{proof-ir8/fig}
\end{figure}

\begin{figure}[htb]
\centering \fig{}{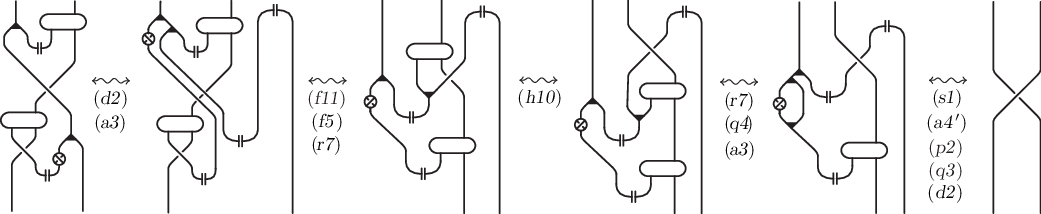}
\caption{Proof that in $\overline{\overline{\H}}$$^r$ $(q8)$  follows from $(h10)$.}
\label{proof-hr9/fig}
\end{figure}


\begin{thebibliography}{99}
\bibitem{AS} M. Asaeda, {\it Tensor functors on a certain category constructed from spherical categories}, preprint, http://math.ucr.edu/~marta/4\_Dez\_10.pdf

\bibitem{BD21} A. Beliakova,  M. De Renzi, {\it Kerler - Lyubashenko Functors on 4D 2-Handlebodies}, IMRN, rnac039, (2022)

\bibitem{BBDP} A. Beliakova, I. Bobtcheva,  M. De Renzi, R. Piergallini, {\it On Algebraisation in Low-Dimensional Topology}, in preparation

\bibitem {BP} I. Bobtcheva, R. Piergallini,  {\it On four-dimensional 2-handlebodies and three-manifolds}, JKTR {\bf 21} (2012)

\bibitem {BM02} I. Bobtcheva, M. G. Messia,  {\it HKR-type invariants of 4-thickenings of 2-dimensional CW complexes}, Algebraic \& Geometric Topology  {\bf 3} (2002)

\bibitem{CY99} L. Crane, D. Yetter, {\it On algebraic structures implicit in topological quantum field theories}, J. Knot Theory Ramifications 8 (1999) no. 2, 125-163

\bibitem{GS99} R.E. Gompf and A.I. Stipsicz, {\it 4-manifolds and Kirby calculus}, Grad.
Studies in Math. {\bf 20}, Amer. Math. Soc. 1999.

\bibitem{Ha00} K. Habiro, {\it Claspers and finite type invariants of links}, Geometry \&
Topology, {\bf 4} (2000), 1--83.

\bibitem{H96} M.Hennings,
     {\it Invariants from links and 3-manifolds obtained from Hopf  algebras},  J.London Math.Soc. (2) {\bf 54}  (1996), 594-624.

\bibitem{Ke99} T. Kerler, {\it Bridged links and tangle presentations of cobordism
categories}, Adv. Math {\bf 141} (1999), 207--281.

\bibitem{Ke02} T. Kerler, {\it Towards an algebraic characterization of 3-dimensional
cobordisms}, Contemporary Mathematics {\bf 318} (2003), 141--173.

\bibitem{KL01} T. Kerler and V.V. Lyubashenko, {\it Non-semisimple topological quantum
field theories for 3-manifolds with corners}, Lecture Notes in Mathematics {\bf 1765},
Springer-Verlag 2001.

\bibitem{Ki78} R. Kirby, {\it A calculus for framed links in $S^3$}, Invent. math. {\bf
45} (1978), 36--56.

\bibitem{Ki89} R. Kirby, {\it The topology of 4-manifolds}, Lecture Notes in Mathematics
{\bf 1374}, Springer-Verlag 1989.

\bibitem{KR95} L.Kaufmann, D.Radford,
     {\it Invariants of 3-manifolds derived from  finite-dimensional Hopf algebras},
     Journal of knot theory and its ramifications {\bf 4}, no. 1
      (1995)
      
\bibitem{L93} G.Lusztig,
  {\it Introduction to quantum groups}, Birkh\"auser Boston, 1993.

\bibitem{Ma91} S. Majid, {\it Braided Groups and Algebraic Quantum Field Theories}, Lett. Math. Phys. 22, no. 3, (1991), 167--175.

\bibitem{M93} S. Montgomery, {\it Hopf algebras and their action on rings}, Regional Conference Series in Mathemaics {\bf
82} (1993)

\bibitem{Oh02} T. Ohtsuki, {\it Problems on invariants of knots and 3-manifolds}, Geom.
Topol. Monogr. {\bf 4}, in ``Invariants of knots and 3-manifolds (Kyoto, 2001)'', Geom.
Topol. Publ. 2002, 377--572.

\bibitem{RT91} N.Yu.Reshetikhin and V.G.Turaev,
   {\it Invariants of 3-manifold via link polynomials and quantum groups},
   Invent.Math. {\bf 103} (1991),547-597.

\bibitem{S69} M.E. Sweedler, {\it Integrals for Hopf Algebras}, Ann. Math. 89 (1969), 323-335.

\end{thebibliography}
\end{document}